\documentclass{article}
\usepackage[utf8]{inputenc}
\usepackage{tikz}
\usepackage{fullpage}
\usepackage{hyperref,verbatim}
\usepackage{enumitem,bm}
\usepackage{amsthm,amsmath,amssymb}

\newtheorem{defi}{Definition}

\newtheorem{cor}[defi]{Corollary}
\newtheorem{thr}[defi]{Theorem}
\newtheorem{lem}[defi]{Lemma}

\newtheorem{prop}[defi]{Proposition}

\newtheorem{remark}[defi]{Remark}
\newtheorem{claim}[defi]{Claim}
\newcommand*{\myproofname}{Proof}
\newenvironment{claimproof}[1][\myproofname]{\begin{proof}[#1]}{\end{proof}}
\def\C{\mathcal{C}}

\title{Extremal entropy for graphs with given size}

\author{
Stijn Cambie\thanks{
Mathematics Institute, University of Warwick, UK. E-mail: {\tt stijn.cambie@hotmail.com}. Part of this research was supported by the UK Research and Innovation Future Leaders Fellowship MR/S016325/1.}\and
Matteo Mazzamurro\thanks{
Department of Computer Science,
University of Warwick, UK. E-mail: {\tt  matteo.mazzamurro@warwick.ac.uk}. This paper is partly funded by: EPSRC Centre for Doctoral Training in Urban Science and Progress (EP/L016400/1), EPSRC DTP (EP/N509796/1).}
}

\begin{document}

\maketitle

\begin{abstract}
    The first degree-based entropy of a graph is the Shannon entropy of its degree sequence normalized by the degree sum. Its correct interpretation as a measure of uniformity of the degree sequence requires the determination of its extremal values given natural constraints. In this paper, we prove that the graphs with given size that minimize the first degree-based entropy are the colex graphs.
\end{abstract}

\section{Introduction}
Several graph invariants are finite sequences of non-negative integers. When normalized by its sum, a finite sequence can be thought of as a discrete probability distribution. Dehmer~\cite{dehmer2008information} proposed a general approach to study the Shannon entropy of probability distributions derived by normalising graph invariants, which he called information functionals. The Shannon entropy of an information functional is an intuitive measure of uniformity of the corresponding graph invariant. Yet, its careful and context-informed normalisation and interpretation poses some challenges. Determining the range of values the entropy can take is a non-trivial task, as it depends on the presence of structural constraints on the graph \cite{dehmer2012extremal}. A key step towards solving this issue is first and foremost the identification of the measure's extremal values for graphs satisfying natural constraints.

A large number of studies have tackled this problem when the information functional is given by powers $d^c$ of the vertex degrees $d$, normalized by their sum. Cao et al.~\cite{cao2014extremality} proved extremal properties for certain classes of graphs (trees, unicyclic, bicyclic, and chemical graphs of given order and size) for the special case of the first degree-based entropy, i.e., when the exponent of the degree powers is $c=1$. In~\cite{cao2015degree}, they extended their work to the case $c>1$ and provided bounds depending on the smallest and largest degree of the graph. The smallest and largest degree were similarly used in \cite{lu2015some} and \cite{lu2016new} to provide new bounds, which they proved via Jensen's inequality. Chen et al.~\cite{chen2015bounds}, instead, focused on the relation between the entropy and different values of the exponent $c$, providing numerical results for trees, unicyclic, bipartite and triangle-free graph with small number of vertices. Das and Dehmer~\cite{das2016conjecture} conjectured that the path graph maximizes the first degree-based entropy among trees, a result eventually proved in \cite{ilic2016extremal}. Ghalavand at al.~\cite{ghalavand2019first} focused again on the first degree-based entropy and used the Strong Mixing Variable method to extend maximality results for trees, unicyclic and bicyclic graphs, following an idea very similar to that used in \cite[Lemma~1]{cao2014extremality}. Yan~\cite{yan2021topological} focused on the first degree-based entropy for graphs with given order and size and proved that a graph whose degree sequence achieves minimum entropy must be a threshold graph (see Lemma~\ref{lem:extr_is_threshold_graph}). 
The same result holds true when only the graph size is fixed.

In this paper, we find the graphs yielding minimum and maximum first degree-based entropy among graphs with given size. We begin by formally introducing the definition of first degree-based entropy. 

\begin{defi}
    The first degree-based entropy of a graph $G$ with degree sequence $(d_i)_{1 \le i \le n}$ and size $m$ equals
    $$I(G)=-\sum_{i=1}^n \frac{d_i}{2m}\log\left( \frac{d_i}{2m} \right).$$
\end{defi}

By introducing the help function $h(G)=\sum_i d_i \log(d_i)=\sum_i f(d_i),$
    where $f(x)=x \log(x),$
    we have $I(G)=\log(2m)-\frac{1}{2m}h(G).$
    Hence, determining the minimum (resp. maximum) of $I(G)$ is equivalent to determining the maximum (resp. minimum) of $h(G).$


Among all graphs of size $m$, $m K_2$ is the (unique) graph attaining the maximum entropy, since $h(G) \ge 0=h(m K_2).$
The analogous question about determining the minimum entropy turns out to be harder. Here we prove that the colex graph $\C(m)$ is extremal.
The colex graph $\C(m)$ maximizes the number of triangles (and cliques) among all graphs with size $m$ as a corollary of the Kruskal-Katona theorem~\cite{K63,K66}, as observed in~\cite{KR}. Thus it can be considered as the most clustered graph and so a natural candidate for minimizing entropy (which is a measure of disorder).

When $m=\binom{k}{2}+\ell$ where $0 \le \ell <k$, the colex graph is the graph formed by connecting a vertex $v$ to $\ell$ vertices of a clique $K_k.$ An example has been depicted in Figure~\ref{fig:colexgraph} for $m=31,$ i.e. $k=8$ and $\ell=3.$

\begin{figure}[ht]
\centering
\begin{tikzpicture}{
\foreach \x in {0,45,...,315}{\draw[fill] (\x+22.5:1.2) circle (0.1);}
\foreach \x in {0,45,90}{
\draw(\x-22.5:1.2) -- (0:2.4);}
\foreach \x in {0,45,90,135}{\draw[thick] (\x+22.5:1.2) -- (\x+202.5:1.2);}
\foreach \x in {0,45,...,315}{\draw[thick] (\x+22.5:1.2) -- (\x+67.5:1.2);}
\draw[thick] (22.5:1.2) -- (112.5:1.2) -- (202.5:1.2) -- (292.5:1.2) -- cycle;
\draw[thick] (67.5:1.2) -- (157.5:1.2) -- (247.5:1.2) -- (337.5:1.2) -- cycle;
\draw[thick] (22.5:1.2) -- (157.5:1.2) -- (292.5:1.2) -- (67.5:1.2) -- (202.5:1.2) -- (337.5:1.2) -- (112.5:1.2) -- (247.5:1.2) -- cycle;
\draw[fill] (0:2.4) circle (0.1);}
\end{tikzpicture}
\caption{The graph $\C(31)$}\label{fig:colexgraph}
\end{figure}

Now we can formally state our main result.
\begin{thr}\label{thr:main}
    Among all graphs with size $m$, the colex graph $\C(m)$ maximizes $h(G)$ and therefore minimizes the entropy.
\end{thr}

The main ideas for the proof are given in Section~\ref{sec:main}.
The proof relies on some known results, which are listed in Section~\ref{sec:def+res} and on a Lemma, verified in Section~\ref{sec:proof_comp_lemma}.

\section{Preliminary definitions and results}\label{sec:def+res}

In this section, we mention a few definitions and results from the past, which will be applied later. We start with the definition of majorizing sequences. When we speak about sequences majorizing each other, we always compare non-increasing sequences (equivalently, we order the sequence) of nonnegative reals. If necessary, we add zeros to be sure they are the same length.

\begin{defi}
    A sequence $(x_i)_{1 \le i \le n}$ majorizes the sequence $(y_i)_{1 \le i \le n}$ if and only if  $\sum_{1 \le i \le j } x_i \ge \sum_{1 \le i \le j } y_i$ for every $1 \le j \le n$ and equality does hold for $j=n$.
\end{defi}

This definition is used in the following useful inequality, which we immediately apply to the following corollary, lemma, and proposition. The corollary has also been observed in~\cite{E18}.

\begin{thr}[Karamata's inequality, \cite{Kar32}]\label{thr:Kar}
    Let $(x_i)_{1 \le i \le n}$ be a sequence majorizing the sequence $(y_i)_{1 \le i \le n}$. Then for every convex function $g$ we have 
	$\sum_{1 \le i \le n } g(x_i) \ge \sum_{1 \le i \le n } g(y_i).$ Furthermore this inequality is strict if the sequences are not equal and $g$ is a strictly convex function. For concave functions, the same holds with the opposite sign.
\end{thr}

\begin{cor}\label{cor:majdeg_kar}
    Let $G$ and $G'$ be two graphs such that the degree sequence of $G$ majorizes the degree sequence of $G'$. 
    Then $h(G) \ge h(G').$
\end{cor}

\begin{proof}
    By Theorem~\ref{thr:Kar}, it is sufficient to note that $f(x)=x \log (x)$ is a convex function.
\end{proof}


\begin{lem}\label{lem:max_with_balanced}
    Let $t,\ell>0$ and 
    be fixed values.
    Then under the condition that $\sum_{j=1}^n z_i =t$ and all $z_i \ge 0$, $ \sum_{j=1}^n f(z_j+\ell) -\sum_{j=1}^n f(z_j)$ is maximized when $z_1=z_2=\ldots=z_n=\lfloor \frac tn \rceil.$
\end{lem}

\begin{proof}
    Note that the function $\delta(z)=f(z+\ell) - f(z)$ is a strictly concave function for every $\ell>0$ and every sequence of $n$ integers with sum $t$ majorizes the sequence with $z_1=z_2=\ldots=z_n=\lfloor \frac tn \rceil.$
    Now the result follows immediately by Karamata's inequality.
\end{proof}

\begin{prop}\label{prop:basiccases_trees}
    Among trees of order $n$, $h(G)$ is maximized by the star $S_n$ and minimized by the path $P_n$.
   The second smallest value is attained by any tree with $3$ leaves.
\end{prop}
\begin{proof}
    The degree sequence of a tree satisfies $\sum_{1 \le i \le n} d_i =2(n-1)$ by the hand shaking lemma.
    Since a tree has at least $2$ leaves and all degrees are at least $1$, the degree sequence majorizes $(2,2,\ldots,2,1,1)$ and is majorized by $(n-1,1,1,\ldots 1).$
    The path and the star are the only graphs with this degree sequence.
    Any tree that is not a path has at least $3$ leaves.
    If it has more than $3$ leaves, 
    its degree sequence majorizes $(3,2,2,2,\ldots,2,1,1,1)$.
    So the conclusion follows from Corollary~\ref{cor:majdeg_kar}.
\end{proof}
Note that this statement is also true for $h_g(G)= \sum_i g(d_i)$ for any convex function $g$, in particular this proves~\cite[Conj.~1]{E18}.

Finally, we state the following result.

\begin{lem}\label{lem:extr_is_threshold_graph}
Let $G$ be a graph maximizing $h(G)$ among all graphs of size $m$. Then $G$ is a threshold graph, i.e. the vertices of $G$ can be partitioned into a clique $K$ and a stable set $S$ and the neighborhoods of the vertices of $S$ are nested.
\end{lem}

\begin{proof}
This was proved in \cite[Theorem~4]{yan2021topological} for connected graphs with given order and size.
Since $f(x)+f(y)<f(x+y)$, it is easy to see that the graph $G$ maximizing $h(G)$ among all graphs of size $m$, is connected.
As such it maximizes $h(G)$ among all graphs with size $m$ and order $n(G)$ as well, so $G$ is indeed a threshold graph.

\end{proof}

\section{Minimum entropy graphs given size}\label{sec:main}

In this section, we characterize the extremal graphs $G$ maximizing $h(G)$ given their size $m$ and clique number $k$. 
The extremal graphs are related to the colex graph, which as a corollary will be the graph minimizing the entropy among all graphs with size $m.$
\begin{defi}
Let $k$ and $m$ be integers with $m=\binom{k-1}{2}+a(k-1)+b$ for some integers $a,b$ satisfying $a\ge 0$ and $0 \le b \le k-2$.
The graph $\C(m,k)$ is formed by a clique $K_{k-1}$ whose vertices are all connected to a stable set of size $a$ and with $b$ vertices of a clique connected to one additional vertex. 
\end{defi}

An example of $\C(m,k)$ graph is presented in Figure~\ref{fig:adaptedcolexgraph}.

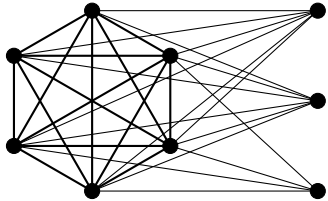
\begin{figure}[ht]
\centering
\begin{tikzpicture}{
\foreach \x in {0,60,...,300}{\draw[fill] (\x+30:1.2) circle (0.1);}
\foreach \x in {0,60,...,300}{
\draw(\x-30:1.2) -- (3,1.2);}
\foreach \x in {0,60,...,300}{
\draw(\x-30:1.2) -- (3,0);}
\foreach \x in {0,-60,-120,60}{
\draw(\x-30:1.2) -- (3,-1.2);}
\foreach \x in {0,120,240}{\draw[thick] (\x+30:1.2) -- (\x+210:1.2);}
\foreach \x in {0,60,...,300}{\draw[thick] (\x+30:1.2) -- (\x+90:1.2);}
\foreach \x in {0,60,...,300}{\draw[thick] (\x+30:1.2) -- (\x+150:1.2);}
\draw[fill] (0:3) circle (0.1);
\draw[fill] (3,1.2) circle (0.1);
\draw[fill] (3,-1.2) circle (0.1);}
\end{tikzpicture}
\caption{The graph $\C(31,7)$} \label{fig:adaptedcolexgraph}
\end{figure}

\begin{remark}\label{remark:colex}
When $m<\binom{k+1}{2}$, the graph $\C(m,k)$ is just the colex graph $\C(m)$.
\end{remark}

For instance, the graph $\C(31,8)$ is just the colex graph $\C(31)$ in Figure \ref{fig:colexgraph}.

The following lemma will be handy in both the theorem and its corollary. The proof of this lemma is quite computational and thus postponed to Section~\ref{sec:proof_comp_lemma}.
\begin{lem}\label{lem:largeclique}
    Let $k\ge 3$. For every $m\ge \binom{k}{2},$ we have $h(\C(m,k)) > h(\C(m,k-1)).$
\end{lem}

\begin{thr}\label{thr:extremalthresholdgraph_givencliquenr}
    Let $k$ and $m$ be integers with $m=\binom{k-1}{2}+a(k-1)+b$ for some integers $a,b$ satisfying $a>0$ and $0 \le b \le k-2$.
    Among all threshold graphs $G$ with clique number $k$ and size $m$, $h(G)$ is maximized by the graph $\C(m,k)$.  
\end{thr}
\begin{proof}
    We will prove this by a double induction on $k$ and $m$.
    When $k=2$, the graph is a tree and the unique extremal graph is a star, as mentioned in Proposition~\ref{prop:basiccases_trees}. So the base case with $k=2$ is true for every $m \ge 1.$
    Now we assume that the statement of the theorem has been proven for all values from $2$ to $k-1$ for every choice of $m$ and prove by induction that this hypothesis is also true for clique number $k$.
    When $m=\binom{k}{2},$ the graph $\C(m,k)=K_k$ is the unique one and thus extremal.
    So now assume it is proven when the size is between $\binom{k}{2}$ and $m-1$ and let $G$ be a threshold graph with clique number $k$ and size $m$.
    Let $v$ be a vertex in the clique whose degree is precisely $k-1.$
    Let $d_1, d_2, \ldots,d_{k-1}$ be the degrees of the $k-1$ neighbours of $v$.
    In that case, we have
    \begin{equation}\label{eq:hG_inductive-gen}
        h(G)= h(G \backslash v)  +f(k-1)+\sum_{j=1}^{k-1} f(d_j) -\sum_{j=1}^{k-1} f(d_j-1)
    \end{equation}
    When $m< \binom{k}{2}+k-1$, $\omega(G \backslash v)=k-1$ and since the result is known for clique number $k-1$, we have $h(G \backslash v) \le h(\C(m-(k-1),k-1)).$
    When $m\ge \binom{k}{2}+k-1$, by the induction hypothesis and lemma~\ref{lem:largeclique} we know that $h(G \backslash v) \le h(\C(m-(k-1),k))$.
    By Lemma~\ref{lem:max_with_balanced} the other part of the upper bound in~\eqref{eq:hG_inductive-gen} is maximized if and only if all $d_i$ differ at most $1$.
    Thus, we conclude that the maximum is precisely $h(\C(m,k))$, and that $\C(m,k)$ is the unique extremal graph.
    So by induction, the statement is true for every $m.$ Finally, by complete induction we conclude that the theorem is true for every value of $k.$
\end{proof}

As a corollary, we now derive the proof of our main result.
\begin{proof}[Proof of Theorem~\ref{thr:main}]
    We know that the extremal graph is a threshold graph by Lemma~\ref{lem:extr_is_threshold_graph} and therefore the result follows from Theorem~\ref{thr:extremalthresholdgraph_givencliquenr}, Lemma~\ref{lem:largeclique} and Remark~\ref{remark:colex}.
\end{proof}

\section{Proof of Lemma~\ref{lem:largeclique}}\label{sec:proof_comp_lemma}

\subsection{Notation and remarks}

Write 
        \begin{equation*}
            m= \binom{k-1}{2} + a(k-1) +b,
        \end{equation*}
        with $a\ge 0$ and $0 \leq b \leq k-2$, and
        \begin{equation}\label{eq:expression_m}
            m= \binom{k-2}{2} + a'(k-2) +b',
        \end{equation}
        with $a'\ge 2$ and $0 \leq b' \leq k-3$.
        Then the degree sequence of $C(m,k)$ is \begin{equation*}
        ((k-1+a)^b,(k-2+a)^{k-1-b},(k-1)^a,b),
        \end{equation*}
        and the degree sequence of $C(m,k-1)$ is 
        \begin{equation*}
        ((k-2+a')^{b'},(k-3+a')^{k-2-b'},(k-2)^{a'},b').
        \end{equation*}
       Remark that since 
       $m= \binom{k-1}{2} + a(k-1) +b = \binom{k-2}{2} +(k-2)+a(k-2)+a+b$, we have
        
        \begin{align*}
            a'&=a+1+\left\lfloor \frac{a+b}{k-2} \right\rfloor \mbox{ and}\\
            b'&\equiv a+b \pmod{k-2}.
        \end{align*}
 
Also note that \begin{equation*}
    f(x) = x \log(x) = \int_1^x \left( \log(x)+1\right) dx,
\end{equation*} a fact that we use frequently in the proof.

\subsection{Summary of the proof}

Our proof of Lemma~\ref{lem:largeclique} involves three steps aimed at reducing the general case to a more manageable one:
\begin{enumerate}
    \item We begin by showing that if Lemma~\ref{lem:largeclique} is true for all $m$ for which $b'=0$ in equation~\eqref{eq:expression_m}, then it holds for all $m$, i.e. for every $0\le b'\leq k-3$ as well. We do so by showing that if a counterexample does exist for some $m$ with $1\le b'\leq k-3$, then there is a counterexample with size $m'$ which can be represented in such a way that the corresponding $b'$ (in~\eqref{eq:expression_m}) is $0$. 
    \item We then show that if Lemma~\ref{lem:largeclique} holds in the case $b'=0$ and $2\le a'\leq k$, then it holds for all values of $a'$. We do so by proving that if our claim is true for a certain size $m$, then so it is for $m'=m+(k-1)(k-2).$
    \item We conclude by proving that the claim holds for values $m$ that be written as in~\eqref{eq:expression_m} with $2 \le a' \leq k$ and $b'=0.$
\end{enumerate}

In the last two steps, we are careful with a reduction to $m=\binom k2 -1,$ since then $\C(m,k)=\C(m,k-1)$.

In the following, we state these steps in the form of claims, which are proven separately. Note that one can check the three claims in reverse order as well. In that case, one is increasing the range of values of $m$ for which it is known that Lemma~\ref{lem:largeclique} is true. 

\subsection{Detailed proof}

    \begin{claim}
        It is sufficient to prove Lemma~\ref{lem:largeclique} for the case $b'=0$ (and every $a' \ge 3$).
    \end{claim}
    \begin{claimproof}
        If $h(\C(m,k)) \le h(\C(m,k-1))$ and $k-2>b'\ge b$, then $h(\C(m+1,k)) \le h(\C(m+1,k-1))$ as well, since 
        \begin{align*}
        h(\C(m+1,k-1))-h(\C(m,k-1))&=f(k-2+a')-f(k-3+a')+f(b'+1)-f(b')\\
        &\ge f(k-1+a)-f(k-2+a)+f(b+1)-f(b)\\
        &= h(\C(m+1,k))-h(\C(m,k)).
        \end{align*}
        So we can repeat this until $b'+1=k-2$.
        But then $m+1=\binom{k-2}{2}+(a'+1)(k+2),$ so it is sufficient to verify a case with $b'=0.$

        If $h(\C(m,k)) \le h(\C(m,k-1))$ and $0<b'< b\le k-2$ (note that this implies that $a'\ge 3$), then $h(\C(m-1,k)) \le h(\C(m-1,k-1))$ as well.
        For this note that $\frac{b}{b'} \ge \frac{k-2}{k-3}$ and $\frac{k-3+a'}{k-2+a} \le \frac{k-1}{k-2}< \frac{k-2}{k-3},$
        the latter being true since $a'-1=a+\left\lfloor \frac{a+b}{k-2} \right\rfloor \le a+\frac{a+k-2}{k-2}$. This implies that 
        \begin{align}
        h(\C(m,k))-h(\C(m-1,k))&= f(k-1+a)-f(k-2+a)+f(b)-f(b-1)\notag\\
        &\ge f(k-2+a')-f(k-3+a')+f(b')-f(b'-1)\label{ineq:caseb'}\\
        &= h(\C(m,k-1))-h(\C(m-1,k-1)).\notag
        \end{align}
        Inequality~\eqref{ineq:caseb'} is true by the following two inequalities 
        \begin{align*}
            f(b)-f(b-1)-\left( f(b')-f(b'-1) \right) &= \int_0^1 \log\left( \frac{b-1+t}{b'-1+t} \right) dt\\
            &\ge \int_0^1 \log \left(\frac{b}{b'}\right) dt\\
            &= \log\left( \frac{b}{b'} \right)\\
            &\ge \log\left( \frac{k-2}{k-3} \right)
        \end{align*}
        and 
        \begin{align*}
            f(k-2+a')-f(k-3+a')-\left( f(k-1+a)-f(k-2+a) \right) &= \int_0^1 \log\left( \frac{k-3+a'+t}{k-2+a+t} \right) dt\\
            &\le \int_0^1 \log\left( \frac{k-3+a'}{k-2+a} \right) dt\\
            &\le \log\left( \frac{k-2}{k-3} \right).
        \end{align*}
    \end{claimproof}
    
    \begin{claim}
        It is sufficient to prove Lemma~\ref{lem:largeclique} for the case $2\le a'\le k-1$ and $b'=0$.
    \end{claim}
    \begin{claimproof}
    
    It is sufficient to prove that for every $m \ge \binom{k}{2}-1$ and $m'=m+(k-1)(k-2),$ we have
    $h(\C(m',k))-h(\C(m',k-1))\ge h(\C(m,k))-h(\C(m,k-1))$ and that this inequality is strict when $m=\binom{k}{2}-1.$
    
    Write $m=\binom{k-2}{2}+a'(k-2)=\binom{k-1}{2}+a(k-1)+b$ and let $m'=m+(k-1)(k-2)=\binom{k-2}{2}+(a'+k-1)(k-2)$.
        
    Then
    \begin{align*}
        h(\C(m',k))-h(\C(m,k))
        =&b \left( f(k-1+a+k-2)-f(k-1+a) \right) \\ &+ (k-1-b)\left( f(k-2+a+k-2)-f(k-2+a) \right) + (k-2)f(k-1)   \\
        \ge & 
        (k-1)\left( f(k-2+a+k-2)-f(k-2+a) \right) + (k-2)f(k-1)
    \end{align*}
    since $f(k-1+a+k-2)-f(k-1+a) > f(k-2+a+k-2)-f(k-2+a).$ Also,
    \begin{align*}
        h(\C(m',k-1))-h(\C(m,k-1))&=
        (k-2)\left( f(k-3+a'+k-1)-f(k-3+a') \right) + (k-1)f(k-2).
    \end{align*}
        
        Here $k-3+a'\le \frac{k-1}{k-2}(k-2+a).$
        Hence, we have 
        
        \begin{align*}
            f(k-2+a+k-2)-f(k-2+a)-(k-2)  &= \int_0^{k-2}  \log\left( k-2+a +t \right) dt\\
            &=\frac{k-2}{k-1} \int_0^{k-1}  \log\left( k-2+a +\frac{k-2}{k-1} t \right) dt
        \end{align*}
        and
        \begin{align*}
            f(k-3+a'+k-1)-f(k-3+a')-(k-1)  &= \int_0^{k-1}  \log\left( k-3+a' +t \right) dt\\
            &\le \int_0^{k-1}  \log\left( \frac{k-1}{k-2}(k-2+a) +t \right) dt\\
            &= \int_0^{k-1}  \log\left( k-2+a +\frac{k-2}{k-1} t \right) +\log\left( \frac{k-1}{k-2}\right) dt.
        \end{align*}
    So we conclude $h(\C(m',k))-h(\C(m,k)) \ge h(\C(m',k-1))-h(\C(m,k-1))$, which is equivalent to $h(\C(m',k))-h(\C(m',k-1)) \ge h(\C(m,k))-h(\C(m,k-1)).$
    Furthermore the difference was strict when $b>0$, which is the case when $m=\binom k2 -1.$
    \end{claimproof}

    \begin{claim}
         Lemma~\ref{lem:largeclique} is true when $2\le a' \le k$ and $b'=0$, i.e.\ when $m=\binom{k-2}{2}+(a+2)(k-2)=\binom{k-1}{2}+a(k-1)+(k-2-a)$ for some $0 \le a \le k-2.$
    \end{claim}
    \begin{claimproof}
        For this, we need to prove that 
        \begin{equation*}
            (k-2)f(k+a-1)+(a+2)f(k-2) \le (k-2-a)f(k+a-1)+(a+1)f(k-2+a)+a f(k-1) +f(k-2-a)
        \end{equation*}
        and that this is strict when $a>0$.
        The latter is equivalent to
        \begin{equation}
             a \left[ f(k+a-1)-f(k+a-2)-f(k-1)+f(k-2) \right]\le f(k-2+a)+f(k-2-a)-2f(k-2).\label{ineq:telescoping}
        \end{equation}
        We rewrite the right hand side of inequality~\eqref{ineq:telescoping} as a telescoping sum:
        \begin{align*}
           &f(k-2+a)+f(k-2-a)-2f(k-2) \\= &\sum_{i=1}^a 
            \left[ f(k-2+i)-f(k-3+i)-f(k-2-a+i)+f(k-3-a+i) \right].
        \end{align*}
        The inequality now follows from the fact that 
        \begin{align*}
            f(x+1)-f(x)-f(x+1-a)+f(x-a)
            =\int_{0}^1 \log \left( \frac{ x+t}{x-a+t} \right) dt
        \end{align*} is a function that is strictly decreasing in $x$ (where the domain is $x \ge a$).
    \end{claimproof}

\section{Concluding remarks}

In this paper, we proved that the graph minimizing the entropy among all graphs with size $m$ is the colex graph $\C(m).$
The colex graph is the graph that maximizes the number of triangles as well, so it is a structured graph for which one might expect that the entropy is small.
We remark that Gan-Loh-Sudakov (\cite{GLS}) type problems with a maximum degree $r$ condition for the entropy-question does not give any additional challenge, in contrast with the clique-version~\cite{CdK19,C19,CC20}.
When $m \le \binom{r+1}{2}$, $\C(m)$ is the extremal graph.
When $m>\binom{r+1}{2}$, any graph all of whose degrees are equal to $r$ except from possibly one is extremal by Corollary~\ref{cor:majdeg_kar}.

\bibliographystyle{abbrv}
\bibliography{bib_entropy}

\end{document}